\documentclass[12pt,
reqno]{amsart}

\usepackage[cp1251]{inputenc}
\usepackage{amsmath,amsxtra,amssymb,eufrak}%{amscd}
\usepackage[all]{xy}
\usepackage[active]{srcltx} %SRC Specials for DVI Searching
\usepackage{mathrsfs}
\usepackage[english]{babel}

\newtheorem{thm}{Theorem}
\newtheorem{lm}[thm]{Lemma}

\newtheorem{pr}[thm]{Proposition}

\theoremstyle{definition}

\newtheorem{rem}[thm]{Remark}

\newcommand{\Rad}{\mathop{\mathrm{Rad}}\nolimits}

\let \al         =\alpha

\let \ep         =\varepsilon      

\let \te         =\theta        % \th нельзя, т.к. это th=sh/ch

\let \phi         =\varphi

\title{When a completion of the universal enveloping algebra is a Banach PI-algebra?}

\author{O. Yu. Aristov}
\email{aristovoyu@inbox.ru}

\begin{document}
 \maketitle
\begin{abstract}
 We prove that a Banach algebra $B$ that is a completion of the universal enveloping algebra of a finite-dimensional complex Lie algebra $\mathfrak{g}$ satisfies a polynomial identity if and only if the nilpotent radical $\mathfrak{n}$ of $\mathfrak{g}$ is associatively  nilpotent in $B$. Furthermore, this holds if and only if a certain polynomial growth condition is satisfied on~$\mathfrak{n}$.
\end{abstract}

\markright{Completion of the universal enveloping algebra}

\subsection*{Introduction}
The theory of representations of a finite-dimensional complex Lie algebra~$\mathfrak{g}$ on a Banach space and, more generally, Lie-algebra homomorphisms from $\mathfrak{g}$ to a Banach algebra unexpectedly turned out to be not as trivial as it might seem at first glance (see the book \cite{BS01}). Indeed, in the solvable case all irreducible representations of~$\mathfrak{g}$ are one-dimensional. Moreover, it was proved by J.~Taylor that in the semisimple case all (not only irreducible) representations of~$\mathfrak{g}$ are finite dimensional \cite{T2}. Nevertheless, a general representation on a Banach space can be quite complicated even when $\mathfrak{g}$ is nilpotent.

Here we consider arbitrary finite-dimensional Lie algebras but restrict ourself to representations (and homomorphisms) with range generating a Banach algebra satisfying a polynomial identity (a PI-algebra). We give an answer for the question: \emph{When does a Banach algebra that is a completion of the universal enveloping algebra  $U(\mathfrak{g})$ satisfy a \emph{PI}} by providing several necessary and sufficient conditions (algebraic and analytic) in terms of the nilpotent radical~$\mathfrak{n}$ of~$\mathfrak{g}$ (Theorems~\ref{1crit} and~\ref{forHom}).  We also include some examples, which show how the criterion works in concrete cases.

According to the nature of the conditions, the argument can be divided in two parts, algebraic and analytic.
The proof of the algebraic part is based on two results in the algebraic PI-theory, a theorem of Bahturin~\cite{Ba85} on PI-quotients in the semisimple case and the Braun-Kemer-Razmyslov theorem on the nilpotency of the Jacobson  radical. In the proof of the analytic part we use  a result of Turovskii~\cite{Tu87} on topological nilpotency, the theorem of Taylor on semisimple Lie algebras mentioned above, and a small piece of the theory of generalized scalar operators.

It is worth noting that this study was initially motivated by non-commutative spectral theory.
In \cite{Do05} Dosi considered a Fr\'echet algebra $\mathfrak{F}_\mathfrak{g}$, `the algebra of formally-radical functions', associated with a nilpotent Lie algebra $\mathfrak{g}$. His goal was applications to non-commutative spectral theory; see \cite{Do09C,Do10A,Do10B,DoSb}. (Dosi's construction is extended to the case of a general solvable Lie algebra in the author's article \cite{ArA06}. It is also shown there that $\mathfrak{F}_\mathfrak{g}$ is an Arens-Michael algebra, i.e., it can be approximated by Banach algebras.)
It is proved in \cite[Lemma~4]{Do05} that in the nilpotent case an embedding $\mathfrak{g}$ to a Banach algebra $B$ can be extended to a continuous  homomorphism $\mathfrak{F}_\mathfrak{g}\to B$ if and only if $\mathfrak{g}$ is supernilpotent, i.e., every element of $[\mathfrak{g},\mathfrak{g}]$ (which is equal to $\mathfrak{n}$ in this case) is associatively nilpotent. This article comes from a desire to understand what form this condition can take in the general case. Our main results show that the property to satisfy a PI, which is formally weaker that the supernilpotency, is reasonable in this context.

This paper is just the beginning of work on PI-completions. As a continuation of the study, we discuss such completions of the algebra of analytic functionals on a complex Lie group in the subsequent articles  \cite{ArPC16,ArPC15}, which  contain the following topics: a generalization of Theorem~\ref{1crit}, a relationship with large-scale geometry and a decomposition into an analytic smash product.

\subsection*{The statement of main results}
Recall that an associative algebra $A$ (in our case over a field, which usually is $\mathbb{C}$) \emph{satisfies a polynomial identity } (in short, a PI-\emph{algebra)} if there is a non-trivial non-commutative polynomial $p$ (i.e., an element of a free algebra in $n$ generators) such that   $p(a_1,\ldots,a_n)=0$ for all $a_1,\ldots,a_n\in A$. PI-algebras can be both unital and non-unital. But we usually assume that associative algebras are unital unless otherwise stated.  Banach PI-algebras are discussed in \cite{Kr87} and \cite{Mu94} but we do not use results contained there. For up-to-date information on general PI-algebras see \cite{AGPR} or \cite{KKR16}.

Recall that the \emph{nilpotent radical} $\mathfrak{n}$ of a Lie algebra $\mathfrak{g}$ is the intersection of the kernels of all irreducible representations of~$\mathfrak{g}$; \cite[Chapter~I, \S\,5.3, p.\,44, Definition~3]{Bou} or \cite[p.\,27, \S\,1.7.2]{Dix}.

The following two theorems are our main results.

\begin{thm}\label{1crit}
Suppose that $\mathfrak{g}$ is a finite-dimensional complex Lie subalgebra of a Banach algebra~$B$. Let $\|\cdot\|$ denote the norm on~$B$ and $\mathfrak{n}$ the nilpotent radical of~$\mathfrak{g}$. If $B$ is generated by $\mathfrak{g}$ as a Banach algebra, then the following conditions are equivalent.

\emph{(1)}~$B$ is a \emph{PI}-algebra.

\emph{(2a)}~Every  element of $\mathfrak{n}$ is nilpotent.

\emph{(2b)}~The non-unital associative subalgebra of~$B$ generated by~$\mathfrak{n}$ is nilpotent.

\emph{(3a)}~$e^b-1$  is nilpotent for every  $b\in\mathfrak{n}$.

\emph{(3b)}~There is $d\in\mathbb{N}$ such that $e^b-1$ is nilpotent of degree at most~$d$ for every $b\in \mathfrak{n}$.

\emph{(4)}~There are $C>0$ and $\al>0$ such that $\|e^b\|\le C (1+\|b\|)^\al$ for every $b\in\mathfrak{n}$.
\end{thm}

Note that the conditions~(2b) and~(3b) are the uniform versions of the point-wise conditions (2a) and~(3a).

In Theorem~\ref{1crit} we assume that $\mathfrak{g}$ is embedded into $B$. Is not hard to extend the result to the case of arbitrary Lie-algebra homomorphism $\mathfrak{g}\to B$ (or, equivalently, a homomorphism $U(\mathfrak{g})\to B$). Namely, the following result holds.

\begin{thm}\label{forHom}
Suppose that $\mathfrak{g}$ is a finite-dimensional complex Lie algebra, $B$ is a Banach algebra and $\te\!:U(\mathfrak{g})\to B$ is a homomorphism of associative algebras. Let $\|\cdot\|$ be the norm on~$B$, $|\cdot|$ a norm on~$\mathfrak{n}$ and $U(\mathfrak{n})_0$ the augmentation ideal of $U(\mathfrak{n})$, i.e., the kernel of the trivial representation. If $\te$ has dense range, then the following conditions are equivalent.

\emph{(1)}~$B$ is a \emph{PI}-algebra.

\emph{(2a)}~$\te(\eta)$ is nilpotent for every  $\eta\in \mathfrak{n}$.

\emph{(2b)}~The non-unital associative algebra $\te(U(\mathfrak{n})_0)$ is nilpotent.

\emph{(3a)}~$e^{\te(\eta)}-1$ is nilpotent for every  $\eta\in \mathfrak{n}$.

\emph{(3b)}~There is $d\in\mathbb{N}$ such that $e^{\te(\eta)}-1$ is nilpotent of degree at most~$d$ for every $\eta\in \mathfrak{n}$.

\emph{(4)}~There are $C>0$ and $\al>0$ such that $\|e^{\te(\eta)}\|\le C (1+|\eta|)^\al$ for every $\eta\in\mathfrak{n}$.
\end{thm}
The proof of Theorem~\ref{forHom} is placed below after the proof of Theorem~\ref{1crit}, to which we now turn.

\subsection*{The algebraic argument}
As mentioned in the introduction, the proof of the theorem is divided in two parts, algebraic and analytic. The algebraic part is in the following proposition.

\begin{pr}\label{algcr}
Let $\mathfrak{g}$ be a finite-dimensional Lie subalgebra of an associative algebra~$A$ over a field of characteristic~$0$, $\mathfrak{s}$ a Levi subalgebra and  $\mathfrak{n}$ the nilpotent radical of~$\mathfrak{g}$. Suppose that $A$ is generated by $\mathfrak{g}$ as an associative  algebra. Then it is a \emph{PI}-algebra if and only if the  unital subalgebra generated by~$\mathfrak{s}$ is finite dimensional and
every  element of $\mathfrak{n}$ is nilpotent.
\end{pr}

For the proof we need a tuple of lemmas. %The first is pure algebraic and holds over a field of characteristic~$0$.

\begin{lm}\label{trrell2}
Let $A$ be an associative algebra over a field, $\mathfrak{g}$ a finite-dimensional Lie subalgebra that generates $A$ as an associative algebra, and~$\mathfrak{n}$  the nilpotent radical of~$\mathfrak{g}$.

\emph{(A)}~Then $\mathfrak{n}$ is contained in the Jacobson radical of $A$.

\emph{(B)}~If, in addition, $A$ is a PI-algebra, then $\mathfrak{n}$ is associatively nilpotent, i.e., there is $d\in\mathbb{N}$ such that $b_1\cdots b_d=0$ for all $b_1,\ldots, b_d\in\mathfrak{n}$.
\end{lm}
The proof of Part~(B) is based on a deep result in the PI-theory,  the Braun-Kemer-Razmyslov theorem on the nilpotency of the Jacobson  radical.
\begin{proof}
It is obvious that $A$ is a quotient of $U(\mathfrak{g})$ and, in particular,  finitely generated.

(A)~Recall that the Jacobson  radical $\Rad A$ of $A$ is the intersection of the kernels of all irreducible representations of~$A$. On the other hand,  the nilpotent radical  $\mathfrak{n}$ of the Lie algebra $\mathfrak{g}$ is the intersection of the kernels of all irreducible representations of~$\mathfrak{g}$.  Since irreducible representations of $U(\mathfrak{g})$ is in one-to-one correspondence with irreducible representations of $\mathfrak{g}$ and every irreducible representation of $A$ can be lifted to an irreducible representation of $U(\mathfrak{g})$, we have that $\mathfrak{n}\subset \Rad A$.

(B)~The Braun-Kemer-Razmyslov theorem asserts that the Jacobson radical of a finitely generated PI-algebra over a commutative Jacobson ring (in particular, over a field)  is nilpotent \cite[p.\,149, Theorem 4.0.1]{KKR16}. This completes the proof.
\end{proof}

A theorem of Wedderburn asserts that a  finite-dimensional  associative  algebra  over  a
field (of arbitrary characteristic) with linear basis  consisting  of  nilpotent  elements is nilpotent; see, e.g., \cite[p.\,56, Theorem 2.3.1]{Her}. We need a more general result in characteristic~$0$: If $\mathfrak{h}$ is a Lie subalgebra of an associative algebra and $\mathfrak{h}$ is generated by finitely many associatively nilpotent elements, then the non-unital associative subalgebra  generated by the solvable radical of $\mathfrak{h}$ is nilpotent \cite[Theorem 8]{HAW}. In particular, we have the following lemma.

\begin{lm}\label{nilte}
Let $A_0$ be a non-unital associative algebra over a field of characteristic~$0$. Suppose that $A_0$ is generated by a nilpotent Lie subalgebra~$\mathfrak{h}$. If every element of $\mathfrak{h}$ is nilpotent, then so is~$A_0$.
\end{lm}

\begin{rem}
In \cite[Lemma 2.2]{Do09C}, Dosi gave a direct proof of the lemma in the partial case when $\mathfrak{h}=[\mathfrak{g},\mathfrak{g}]$ for a nilpotent $\mathfrak{g}$.
\end{rem}

\begin{lm}\label{extPI}
An extension of a \emph{PI}-algebra  \emph{(}over an arbitrary commutative ring\emph{)} by an ideal that is a \emph{PI}-algebra is  also a \emph{PI}-algebra.
\end{lm}
\begin{proof}
Let $I$ be an ideal of an algebra $A$. Suppose that $A/I$  and $I$ satisfy polynomial identities $p$ and $q$, respectively. Let $n$ and $m$ be the numbers of variables in $p$ and $q$, respectively.
If $a_{ij}\in A$ ($i=1,\ldots,n$, $j=1,\ldots,m$), then $p(a_{1j},\ldots,a_{nj})\in I$ for every $j$ because $p(a_{1j}+I,\ldots,a_{nj}+I)\subset I$.
Hence $q(p(a_{11},\ldots,a_{n1}),\ldots,p(a_{1m},\ldots,a_{nm}))=0$. It is easy to see that since $p$ and $q$ are not trivial, this non-commutative polynomial is not trivial. Thus $A$ is a PI-algebra.
\end{proof}

\begin{proof}[Proof of Proposition~\ref{algcr}]
%Denote by $B$ the associative unital subalgebra of $A$ generated by $\mathfrak{g}$.
Denote by $S$ the unital subalgebra of $A$ generated by $\mathfrak{s}$ and suppose that $A$ is a PI-algebra.
Then $S$ is a quotient of $U(\mathfrak{s})$ and satisfies a PI. Since we work in characteristic $0$ and $\mathfrak{s}$ is semisimple, we can apply a result of Bahturin, which asserts that every quotient of $U(\mathfrak{s})$ that is a PI-algebra is finite dimensional (see \cite{Ba85}, Theorem~1 and Corollary). In particular, $S$ is finite dimensional. On the other hand, it follows from Lemma~\ref{trrell2} that every element of $\mathfrak{n}$ is nilpotent. The necessity is proved.

Suppose now that $S$ is finite dimensional and every element of $\mathfrak{n}$ is nilpotent. Denote by $I$ the ideal of~$A$ generated by~$\mathfrak{n}$. It follows from Lemma~\ref{extPI} that it suffices to show that $I$ and $A/I$ are PI-algebras.

Denote by  $A_0$ the non-unital subalgebra of~$A$ generated by~$\mathfrak{n}$  and by $U(\mathfrak{n})_0$  the augmentation ideal of $U(\mathfrak{n})$. Note that $\mathfrak{n}$ is a nilpotent Lie algebra. So by Lemma~\ref{nilte},  there is $d\in\mathbb{N}$ such that $A_0^d=0$. Since $\mathfrak{n}$ is a Lie ideal of~$\mathfrak{g}$, it easily follows that $U(\mathfrak{g})U(\mathfrak{n})_0=U(\mathfrak{n})_0 U(\mathfrak{g})$. Therefore $(U(\mathfrak{g})U(\mathfrak{n})_0)^d\subset U(\mathfrak{n})_0^d U(\mathfrak{g})$ and therefore $I^d=A_0^d A=0$. Hence $I$ is nilpotent and, in particular, a PI-algebra.

On the other hand,  being reductive, $\mathfrak{g}/\mathfrak{n}$ is a direct sum of a semisimple summand and an abelian summand $\mathfrak{a}$
\cite[p.\,27, Proposition 1.7.3]{Dix}. Note that the restriction of the map $\mathfrak{g}\to\mathfrak{g}/\mathfrak{n}$ to $\mathfrak{s}$ is injective and so we can identify the semisimple summand in $\mathfrak{g}/\mathfrak{n}$ with $\mathfrak{s}$. Consider the naturally defined homomorphisms $S\to A\to A/I$ and $U(\mathfrak{a})\to U(\mathfrak{g}/\mathfrak{n})\to A/I$. Their ranges commute since $U(\mathfrak{s})$ and $U(\mathfrak{a})$ commute in $U(\mathfrak{g}/\mathfrak{n})$ and
the diagram
\begin{equation*}
\xymatrix{
\mathfrak{s}\ar[d]\ar[r]&\mathfrak{g}/\mathfrak{n} \ar[d]\\
A \ar[r]& A/I
}
\end{equation*}
is commutative. By the universal property of tensor product of associative algebras,  we have the induced homomorphism  $S\otimes U(\mathfrak{a})\to A/I$. This homomorphism is surjective because $A/I$  and $S\otimes U(\mathfrak{a})$ are quotients of $U(\mathfrak{g})/(U(\mathfrak{g})U(\mathfrak{n})_0)$ and $U(\mathfrak{s})\otimes U(\mathfrak{a})$, respectively, and the last two algebras are isomorphic being  isomorphic  to $U(\mathfrak{g}/\mathfrak{n})$. By Regev’s theorem \cite[p.~138, Theorem 3.4.7]{KKR16}, the class of PI-algebras is stable under tensor products. Thus $S\otimes U(\mathfrak{a})$ is a PI-algebra since it is the tensor product of a finite-dimensional and abelian algebra. Being a quotient of a PI-algebra,  $A/I$ also satisfies a PI. This completes the proof of the sufficiency.
\end{proof}

\subsection*{The analytic argument}

Now we turn to auxiliary lemmas, which are needed in the analytic part of the proof of Theorem~\ref{1crit}.
The first lemma is a corollary of a result of Turovskii in~\cite{Tu87}.

\begin{lm}\label{GRlem}
Suppose that the hypotheses in Theorem~\ref{1crit} hold. Then $\mathfrak{n}\subset \Rad B$.
\end{lm}
\begin{proof}
Note that $\mathfrak{n}=[\mathfrak{g},\mathfrak{r}]$, where $\mathfrak{r}$ is the solvable radical of $\mathfrak{g}$ \cite[p.\,26, Proposition 1.7.1]{Dix}.
Turovskii's theorem \cite{Tu87} asserts that $[\mathfrak{g},\mathfrak{h}]\subset \Rad B$ for every solvable ideal $\mathfrak{h}$ of $\mathfrak{g}$ (for a proof see \cite[\S\,24, p.\,130, Theorem~1]{BS01}).  By putting $\mathfrak{h}=\mathfrak{r}$ we have the result.
\end{proof}

Recall that an element of a Banach algebra is said to be \emph{topologically nilpotent} if $\|b^n\|^{1/n}\to0$ as $n\to \infty$.

\begin{lm}\label{nilpex}
Let $b$ be a topologically nilpotent element of a Banach algebra. Then $b$ is nilpotent if and only if $e^b-1$ is nilpotent. Moreover, $b$ is of degree of nilpotency at most~$d$ if and only if so is $e^b-1$.
\end{lm}
\begin{proof}
Put $r\!:=e^b-1$.
If $b^d=0$ for some $d\in\mathbb{N}$, then $r=b+\cdots+ b^{d-1}/(d-1)!$  and therefore $r^d=0$.

Suppose now that there is $d\in\mathbb{N}$ such that $r^d=0$. Since $b$ is topologically nilpotent element, its  spectrum is $\{0\}$. Let  $f$ be the function on $\mathbb{C}$ such that  $f(z)=(e^z-1)/z$ when $z\ne 0$ and $f(0)=1$. Since it is  holomorphic, it follows from  the spectral mapping theorem \cite[Theorem 2.2.23]{X2} that the spectrum of $f(b)$ is $\{1\}$ and therefore  $f(b)$ is invertible.
Since $bf(b)=e^b-1=r$ and $r^d=0$, we have $b^df(b)^d=0$. Since $f(b)$ is invertible, $b^d=0$.
\end{proof}

\begin{lm}\label{eqpogr}
Let $r$ be a topologically nilpotent element of a Banach algebra with the norm $\|\cdot\|$. Then the following conditions are equivalent.

\emph{(1)}~There are $C>0$ and $\al>0$ such that $\|(1+r)^k\|\le C(1+ |k|)^\al$ for all $k\in\mathbb{Z}$.

\emph{(2)}~$r$ is nilpotent.

Moreover, we can assume that $C$ and $\al0$  depend only on  $\|r\|$  and  the degree of nilpotency of~$r$.
\end{lm}

Note that $(1+r)^k$ is well defined for negative $k$ because the spectrum of $r$ is $\{0\}$.
\begin{proof}
If (1) holds, then $1+r$ is generalized scalar \cite[p.\,66, Theorem 1.5.12]{LN00}, i.e., admits a $C^\infty$-functional calculus on~$\mathbb{C}$. (This result and the proposition cited below are stated for operators but, in fact, the arguments for them use only the Banach algebra structure.) Shifting by a number does not change the property of being generalized scalar. Thus, $r$ is both topologically nilpotent  and generalized scalar. It follows from \cite[p.\,64, Proposition 1.5.10]{LN00} that then $r$ is nilpotent.

On the other hand, if $r$ is nilpotent, then we immediately have the desired estimate for positive~$k$ with constants $C>0$ and $\al>0$  depending only on the degree of nilpotency of~$r$, say $d$, and $\|r\|$. To prove it for negative~$k$ note that $(1+r)^{-1}=1+r'$  for some nilpotent~$r'$ whose degree of nilpotency and norm depend only on $d$ and $\|r\|$. It follows that~(1) holds with $C$ and $\al$ also depending only on $d$ and $\|r\|$.
\end{proof}

\begin{proof}[Proof of Theorem~\ref{1crit}]
We show that (1)$\Leftrightarrow$(2a) and (2a)$\Rightarrow$(2b)$\Rightarrow$(3b)$\Rightarrow$(4)$\Rightarrow$(3a)$\Rightarrow$(2a).  Let $\mathfrak{s}$ be a Levi subalgebra of~$\mathfrak{g}$, $A$ the associative unital subalgebra of $B$ generated by $\mathfrak{g}$ and $A_0$ the non-unital subalgebra of~$A$ generated by~$\mathfrak{n}$.

(1)$\Leftrightarrow$(2a).
Since $\mathfrak{s}$ is semisimple, it follows from a result of J.~Taylor  \cite{T2} (see also \cite{BS01} or \cite{ArAnF}) that the image of $U(\mathfrak{s})$ in a Banach algebra is finite dimensional. Being dense in $B$, the subalgebra $A$ satisfies a PI if and only if so does $B$. Applying Proposition~\ref{algcr} to~$A$ we conclude that  the conditions~(1) and~(2a) are equivalent.

(2a)$\Rightarrow$(2b). Applying Lemma~\ref{nilte} to $A_0$, we have that it is nilpotent when every element of $\mathfrak{n}$ is nilpotent.

(2b)$\Rightarrow$(3b).
Suppose that $A_0$ is nilpotent, i.e., there is $d\in\mathbb{N}$ such that $A_0^d=0$. In particular,  every $b$ in $\mathfrak{n}$ is of degree of nilpotency at most~$d$  and so is $e^b-1$ by Lemma~\ref{nilpex}.

(3b)$\Rightarrow$(4).
Let $d$ be a positive integer such that $e^b-1$ is nilpotent of degree at most~$d$ for every $b\in \mathfrak{n}$. Applying the implication (2)$\Rightarrow$(1) in  Lemma~\ref{eqpogr}, we have that there are constants $C>0$ and $\al>0$ such that
\begin{equation}\label{erkes}
\|e^{kb}\|\le C(1+ |k|)^\al\qquad \text{for $b\in\mathfrak{n}$ and  $k\in\mathbb{Z}$}
\end{equation}
when $\|e^b-1\|\le 1$. Therefore there is $\ep>0$ such that \eqref{erkes} holds when $\|b\|\le \ep$.

Now fix a non-zero element $b$ of $\mathfrak{n}$. Take $k\in\mathbb{N}$ such that $(k-1)\ep\le \|b\|\le k\ep$ and put $b'\!:=k^{-1}b$. Then $\|b'\|\le \ep$ and
\begin{equation}\label{erkes2}
\|e^{b}\|=\|e^{kb'}\|\le C(1+ k)^\al.
\end{equation}
Since $k\le 1+\ep^{-1}\|b\|$, we have
$$
1+ k\le 2+ \ep^{-1}\|b\|\le 2 \left(1+ \frac{\ep^{-1}}2\right) (1+ \|b\|).
$$
Combining this with \eqref{erkes2}, we have that there is $C'>0$ such that $\|e^b\|\le C' (1+\|b\|)^\al$ and $C'$ is independent in~$b$.

(4)$\Rightarrow$(3a) Suppose that for some $C>0$ and $\al>0$ the inequality $\|e^{b}\|\le C (1+\|b\|)^\al$ holds when $b\in\mathfrak{n}$. We claim that $r\!:=e^b-1$ is topologically nilpotent whenever $b\in\mathfrak{n}$. Indeed,  Lemma~\ref{GRlem} implies that $b\in \Rad B$. Note that $r=\sum_{n=1}^{\infty}{b^n}/{n!}$ and $\Rad B$ is closed. So $r$ is also contained in $\Rad B$. Being an element of the radical of a Banach algebra, $r$~is topologically nilpotent.

Further, it follows from the assumption that
\begin{equation*}%\label{1rkes}
\|(1+r)^k\|=\|e^{kb}\|\le C (1+\|kb\|)^\al\le  C (1+\|b\|)^\al (1+|k|)^\al
\end{equation*}
for every $k\in \mathbb{Z}$. So by the implication (1)$\Rightarrow$(2) in Lemma~\ref{eqpogr}, $r$ is nilpotent.

(3a)$\Rightarrow$(2a).
Suppose that $e^b-1$  is  nilpotent for every  $b\in\mathfrak{n}$.
Then it follows immediately from Lemma~\ref{nilpex} that $b$ is nilpotent.
This completes the proof of Theorem~\ref{1crit}.
\end{proof}

\begin{proof}[Proof of Theorem~\ref{forHom}]
To deduce Theorem~\ref{forHom} from Theorem~\ref{1crit} note that every Lie-algebra homomorphism $\phi\!:\mathfrak{g}_1\to\mathfrak{g}_2$ maps the nilpotent radical of~$\mathfrak{g}_1$ into the nilpotent radical of~$\mathfrak{g}_2$. If, in addition, $\phi$ is surjective,  then so is the  Lie-algebra homomorphism $\mathfrak{r}_1\to\mathfrak{r}_2$ of the solvable radicals; see, e.g., \cite[Lemma~4.10]{ArA06}.
Since the nilpotent radicals of~$\mathfrak{g}_1$ and $\mathfrak{g}_2$
coincide respectively with $[\mathfrak{g}_1,\mathfrak{r}_1]$ and $[\mathfrak{g}_2,\mathfrak{r}_1]$
(see the reference in the proof of Lemma~\ref{GRlem}), we have a surjective Lie-algebra homomorphism of the nilpotent radicals. Thus, the conditions (1), (2a), (2b), (3a) and (3b) are equivalent  since so are the corresponding conditions in Theorem~\ref{1crit}.

The proof of the implication (4)$\Rightarrow$(3a) is a slight modification of the argument for the corresponding implication in Theorem~\ref{1crit}. First note that there is $K>0$ such that
\begin{equation}\label{Keta}
\|\te(\eta)\|\le K|\eta|\quad\text{for every $\eta\in \mathfrak{n}$.}
\end{equation}
Then writing $\te(\eta)$ instead of $b$ and using \eqref{Keta}, we deduce from the inequality
$\|e^{\te(\eta)}\|\le C (1+|\eta|)^\al$ for $\eta\in\mathfrak{n}$ that
$$
\|(1+r)^k\|\le  C (1+|\eta|)^\al (1+|k|)^\al
$$
and then apply Lemma~\ref{eqpogr}.

Finally, if~(1) holds, then so does~(4) in Theorem~\ref{1crit} for elements $b$ of the form $\theta(\eta)$  with $\eta\in\mathfrak{n}$. It follows from~\eqref{Keta} that~(4) in Theorem~\ref{forHom} is also satisfied.
\end{proof}

\subsection*{Some examples}
There are many finite-dimensional Lie algebras such that every completion of the universal enveloping algebra with respect to a submultiplicative prenorm (i.e., every completion that is a Banach algebra) satisfies a polynomial identity. The two simplest cases follow.

\begin{pr}
If a finite-dimensional complex Lie algebra $\mathfrak{g}$ is reductive, then every completion of $U(\mathfrak{g})$ with respect to a submultiplicative prenorm satisfies a \emph{PI}.
\end{pr}
\begin{proof}
Since $\mathfrak{g}$ is reductive, the nilpotent radical  is trivial. Thus by Theorem~\ref{forHom}, every completion of $U(\mathfrak{g})$ with respect to a submultiplicative prenorm satisfies a PI.
\end{proof}

\begin{pr}
Let  $\mathfrak{g}$ be a $2$-dimensional non-abelian  complex Lie algebra. Then every completion of $U(\mathfrak{g})$ with respect to a submultiplicative prenorm satisfies a \emph{PI}.
\end{pr}
\begin{proof}
There is a basis $e_1,e_2$ in $\mathfrak{g}$ such that $[e_1,e_2]=e_2$. Note that $\mathfrak{n}=[\mathfrak{g},\mathfrak{g}]=\mathbb{C} e_2$. Let $B$ be a Banach-algebra completion of $U(\mathfrak{g})$ and $\pi$ the  corresponding Lie-algebra homomorphism from $\mathfrak{g}$ to $B$. It is not hard to see that $\pi(e_2)$ is nilpotent; see, e.g, \cite[Example 5.1]{Pir_qfree}.  Then $B$ satisfies a PI by Theorem~\ref{forHom}.
\end{proof}

On the other hand, there are Lie algebras whose universal enveloping algebra admits a completion that does satisfy a PI. Here we give a partial case.

\begin{pr}
Let  $\mathfrak{g}$ be a non-abelian nilpotent complex Lie algebra $\mathfrak{g}$. Then there is a submultiplicative prenorm on $U(\mathfrak{g})$ such that the completion does not satisfy a \emph{PI}.
\end{pr}
\begin{proof}
Since $\mathfrak{g}$ is non-abelian and nilpotent, there  is a sequence $(\te_n)$ of representations of $U(\mathfrak{g})$ on (finite-dimensional) Hilbert spaces such that the sequence $(d_n)$ of the corresponding degrees of nilpotency of $\te_n(\mathfrak{n})$ is unbounded; see, e.g., \cite[Proposition 4.14 and Lemma 4.16]{ArA06}. Using renormalization, we can assume that the set $\{\|\te_n(e_j)\|\}$, where $e_1,\ldots,e_k$ are algebraic  generators of $\mathfrak{g}$, is bounded for every $j$.
Then the infinite sum $\te\!:=\oplus_{n=1}^\infty \te_n$ is a well-defined representation of $U(\mathfrak{g})$ on a Hilbert space.
Since for every $p\in\mathbb{N}_+$ there are $n\in\mathbb{N}_+$ and $\eta\in\mathfrak{n}$ such that $\te_n(\eta)^p\ne 0$,  the completion of the range of $\te$ does not satisfy a PI by Theorem~\ref{forHom}.
\end{proof}

For a general criterion see the follow-up paper \cite{ArPC15}.

\end{document}